\documentclass[11pt]{article}

\usepackage{epsfig}
\usepackage{graphicx}
\usepackage{amsbsy}
\usepackage{amsmath}
\usepackage{amsfonts}
\usepackage{amssymb}
\usepackage{textcomp}
\usepackage{hyperref}
\usepackage{aliascnt}

\newcommand{\mcm}[3]{\newcommand{#1}[#2]{{\ensuremath{#3}}}} 

\mcm{\tuple}{1}{\langle #1 \rangle}
\mcm{\name}{1}{\ulcorner #1 \urcorner}
\mcm{\Nbb}{0}{\mathbb{N}}
\mcm{\Zbb}{0}{\mathbb{Z}}
\mcm{\Rbb}{0}{\mathbb{R}}
\mcm{\Cbb}{0}{\mathbb{C}}
\mcm{\Qbb}{0}{\mathbb{Q}}
\mcm{\Acal}{0}{\cal A}
\mcm{\Bcal}{0}{\cal B}
\mcm{\Ccal}{0}{\cal C}
\mcm{\Dcal}{0}{\cal D}
\mcm{\Ecal}{0}{\cal E}
\mcm{\Fcal}{0}{\cal F}
\mcm{\Gcal}{0}{\cal G}
\mcm{\Hcal}{0}{\cal H}
\mcm{\Ical}{0}{\cal I}
\mcm{\Jcal}{0}{\cal J}
\mcm{\Kcal}{0}{\cal K}
\mcm{\Lcal}{0}{\cal L}
\mcm{\Mcal}{0}{\cal M}
\mcm{\Ncal}{0}{\cal N}
\mcm{\Ocal}{0}{{\cal O}}
\mcm{\Pcal}{0}{{\cal P}}
\mcm{\Qcal}{0}{{\cal Q}}
\mcm{\Rcal}{0}{{\cal R}}
\mcm{\Scal}{0}{{\cal S}}
\mcm{\Tcal}{0}{{\cal T}}
\mcm{\Ucal}{0}{{\cal U}}
\mcm{\Vcal}{0}{{\cal V}}
\mcm{\Wcal}{0}{{\cal W}}
\mcm{\Xcal}{0}{{\cal X}}
\mcm{\Ycal}{0}{{\cal Y}}
\mcm{\Zcal}{0}{{\cal Z}}
\mcm{\Mfrak}{0}{\mathfrak M}

\mcm{\restric}{0}{\upharpoonright}
\mcm{\upset}{0}{\uparrow}
\mcm{\onto}{0}{\twoheadrightarrow}
\mcm{\smallNbb}{0}{{\small \mathbb{N}}}
\DeclareMathOperator{\preop}{op}
\mcm{\op}{0}{^{\preop}}

{\begin{array}{c}
\setlength{\unitlength}{1em}}%
{\end{array}}

\usepackage{amsthm}

\newcommand{\theoremize}[2]{\newaliascnt{#1}{thm} \newtheorem{#1}[#1]{#2} \aliascntresetthe{#1}}

\theoremstyle{plain}
\newtheorem{thm}{Theorem}[section]
\theoremize{lem}{Lemma}
\theoremize{skolem}{Skolem}
\theoremize{fact}{Fact}
\theoremize{sublem}{Sublemma}
\theoremize{claim}{Claim}
\theoremize{obs}{Observation}
\theoremize{prop}{Proposition}
\theoremize{cor}{Corollary}
\theoremize{que}{Question}
\theoremize{oque}{Open Question}
\theoremize{con}{Conjecture}

\theoremstyle{definition}
\theoremize{dfn}{Definition}
\theoremize{rem}{Remark}
\theoremize{eg}{Example}
\theoremize{exercise}{Exercise}
\theoremstyle{plain}

\usepackage{verbatim}
\usepackage{enumerate}
\usepackage[all]{xy}

\usepackage{subfig}

\title{Embedding simply connected \\ 2-complexes in 3-space\\ \Large II. Rotation systems}

\author{Johannes Carmesin
\medskip 
\\
  {University of Birmingham}
}
\newcommand{\sm}{\setminus}

\DeclareMathOperator{\degree}{deg}

\DeclareMathOperator{\Sbb}{\mathbb{S}}

\newcommand{\Sthree}{$\Sbb^3$}
\newcommand{\scom}{simplicial complex}

\mcm{\Fbb}{0}{\mathbb{F}}

\usepackage{xr}
\begin{document}

\maketitle

\begin{abstract}
 We prove that 2-dimensional simplicial complexes whose first homology group is trivial have  
topological embeddings in 3-space if and only if there are embeddings of their link graphs in 
the plane that are compatible at the edges and they are simply connected.
\end{abstract}

\section{Introduction}

This is the second paper in a series of five papers. In the first paper \cite{3space1} of this 
series we give an overview about this series as a whole. In this 
paper we give combinatorial characterisations for when certain simplicial complexes embed in 
3-space. This completes the proof of a 3-dimensional analogue of Kuratowski's characterisation 
of 
planarity for graphs, started in \cite{3space1}.

\vspace{.3cm}

A (2-dimensional) simplicial complex has a topological embedding in 
3-space if and only if it has a piece-wise linear embedding if and only if it has a differential 
embedding \cite{{Bin59},{Hatcher3notes},{Pap43}}.\footnote{However this is not equivalent to having 
a linear embedding, see \cite{Bre83}, and \cite{mtw_hardness} for further references. }
Perelman proved that every compact simply connected 3-dimensional 
manifold is isomorphic to the 3-sphere \Sthree\ \cite{{{Perelman1}, {Perelman3}, 
{Perelman2}}}.  In 
this paper we use Perelman's theorem to prove a combinatorial characterisation of which simply 
connected simplicial complexes can be topologically embedded into $\Sbb^3$ as follows. 

The \emph{link graph} at a vertex $v$ of a simplicial complex is the graph whose vertices 
are the edges incident with $v$ and 
whose edges are the faces incident with $v$ and the incidence relation is as in $C$, see 
\autoref{fig:intro}.    \begin{figure} [htpb]   
\begin{center}
   	  \includegraphics[height=3cm]{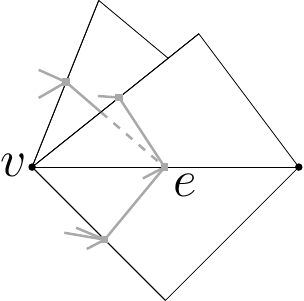}
   	  \caption{The link graph at the vertex $v$ is indicated in grey. The edge 
$e$ projects down to a vertex in the link graph. The faces incident with $e$ 
project down to edges. }\label{fig:intro}
\end{center}
   \end{figure}
Roughly, a \emph{planar rotation system} of a simplicial complex $C$ consists of cyclic 
orientations $\sigma(e)$ of the faces incident 
with each edge $e$ of $C$ such that there are embeddings in the plane of the link graphs 
such that at vertices $e$ 
the cyclic orientations of the incident edges agree with the cyclic orientations $\sigma(e)$.
It is easy to see that if a simplicial complex $C$ has a topological embedding into some oriented 
3-dimensional manifold, then it has a 
planar rotation system. Conversely, for simply connected simplicial complexes the 
existence of 
planar rotation systems is enough to characterise embeddability into \Sthree:

\begin{thm}\label{combi_intro}
 Let $C$ be a simply connected simplicial complex. Then $C$ has a topological embedding into 
$\Sbb^3$ if and only if $C$ has a planar rotation system. 
\end{thm}

A related result has been proved by Skopenkov \cite{Skopenkov94}.
The main result of this paper is the following extension of \autoref{combi_intro}. 

\begin{thm}\label{combi_intro_extended}
 Let $C$ be a simplicial complex such that the first homology group $H_1(C,\Fbb_p)$ is trivial for 
some prime $p$. Then $C$ has a topological embedding into \Sthree\ if and only if $C$ is simply 
connected and it has a planar rotation system. 
\end{thm}

This implies characterisations of topological embeddability into \Sthree\ for the classes of 
simplicial complexes with abelian fundamental group and simplicial complexes in general, see 
\autoref{beyond} for details.

The paper is organised as follows. After reviewing some 
elementary definitions in \autoref{basics}, in \autoref{prelims2}, we 
introduce rotation systems, related concepts and prove basic properties of them. In 
Sections \ref{sec4} and \ref{sec5} we prove \autoref{combi_intro}. 
The proof of \autoref{combi_intro_extended} in \autoref{sec6} makes use of \autoref{combi_intro}. 
Further extensions are derived in \autoref{beyond}. In \autoref{non-or} we discuss how one could 
characterise embeddability of 2-complexes in general 3-manifolds combinatorially. 

\section{Basic definitions}\label{basics}

In this short section we recall some elementary definitions that are important for this paper. 

A \emph{closed trail} in a graph is a cyclically ordered sequence $(e_n|n\in \Zbb_k)$ of distinct 
edges $e_n$ such that the 
starting vertex of $e_{n}$ is equal to the endvertex of $e_{n-1}$.
An (abstract) (2-dimensional) \emph{complex} is a graph\footnote{In this paper graphs are allowed 
to have 
parallel edges and loops.} $G$ together with a 
family of 
closed trails in $G$, called the \emph{faces} of the complex.
We denote complexes $C$ by triples $C=(V,E,F)$, where $V$ is the set of \emph{vertices}, $E$ the 
set of \emph{edges} and $F$ the set of 
faces. We assume furthermore that every vertex of a complex is incident with an edge and every edge 
is incident with a face. 
The \emph{1-skeleton} of a complex $C=(V,E,F)$ is the graph $(V,E)$.
A \emph{directed} complex is a complex together with a choice of 
direction at each of its edges and a choice of orientation at each of its faces.
For an edge $e$, we denote the direction chosen at $e$ by $\vec{e}$. 
For a face $f$, we denote the orientation chosen at $f$ by $\vec{f}$. 

Examples of complexes are (abstract) (2-dimensional) simplicial complexes.
In this paper all simplicial complexes are directed -- although we will not always say it 
explicitly. A \emph{(topological) embedding} of a simplicial complex $C$ into a topological space 
$X$ is an injective 
continuous map from (the geometric realisation of) $C$ into $X$.
In our notation we suppress the embedding map and for example write `$\Sbb^3\sm C$' for 
the topological space obtained from 
$\Sbb^3$ by removing all points in the image of the embedding of $C$. 

In this paper, a \emph{surface} is a compact 2-dimensional manifold (without boundary)\footnote{We 
allow surfaces to be disconnected. }.
Given an embedding of a graph in an oriented surface, the \emph{rotation system} at a vertex $v$ is 
the cyclic orientation\footnote{A \emph{cyclic orientation} is a choice of one of the two 
orientations 
of a cyclic ordering.}  of the edges incident with $v$ given by `walking around' $v$ in the 
surface in a small circle in the direction of the orientation. 
Conversely, a choice of rotation system at each vertex of a graph $G$ defines an embedding of $G$ 
in an oriented surface as explained in \cite{3space1}.

A \emph{cell complex} is a graph $G$ together with a set of directed walks such that each 
direction of an edge of $G$ is in precisely one of these directed walk es. These directed walks are 
called the \emph{cells}. The geometric realisation of a cell complex is obtained from (the 
geometric realisation of) its graph by gluing discs so that the cells are the boundaries of these 
discs. The geometric realisation is always an oriented surface. 
Note that cell complexes need 
not be complexes as cells are allowed to contain both directions of an edge. 
The \emph{rotation system of} a cell complex $C$ is the rotation system of the graph of $C$ in the 
embedding in the oriented surface given by $C$.

\section{Rotation systems}\label{prelims2}

In this section we introduce rotation systems of complexes and some related concepts.

The \emph{link graph} of a \scom\ $C$ at a vertex $v$ is the graph whose vertices are the 
edges incident with $v$.
The edges are the faces incident\footnote{A face is incident with a vertex if 
there is an edge incident with both of them.} with $v$. 
The two endvertices of a face $f$ are those vertices corresponding to the two 
edges of $C$ incident with $f$ and $v$. 
We denote the link graph at $v$ by $L(v)$.

A \emph{rotation system} of a directed complex $C$ consists of for each edge $e$ of $C$ a cyclic 
orientation\footnote{If the edge $e$ is only incident with a single face, then $\sigma(e)$ is 
empty.} $\sigma(e)$ of the faces incident 
with $e$. 

Important examples of rotation systems are those \emph{induced} by topological embeddings of
2-complexes $C$ into \Sthree\  (or more generally in some oriented 3-manifold); here for an edge 
$e$ 
of $C$, 
the cyclic orientation $\sigma(e)$ of the 
faces incident with $e$ is the ordering in which we see the faces when walking around some 
midpoint of $e$ in a circle of small radius\footnote{Formally this means that the circle 
intersects each face in a single point and that it can be contracted onto the chosen midpoint of 
$e$ in such a way that the image of one such contraction map intersects each face in an 
interval.} -- in the direction of the orientation of \Sthree. It can be shown that $\sigma(e)$ is 
independent of the chosen circle if small enough and of the chosen midpoint. 

Such rotation systems have an additional property:
let  $\Sigma=(\sigma(e)|e\in E(C))$ be a rotation system of a simplicial complex $C$ induced by 
a topological embedding of $C$ in the 3-sphere.
Consider a 2-sphere of small diameter around a vertex $v$. We 
may assume that each edge of $C$ intersects this 2-sphere in at most 
one point and that each face intersects it in an interval or not at all. The 
intersection of the 2-sphere and $C$ is a graph: the link graph at $v$. Hence link 
graphs of 2-complexes with rotation systems induced from embeddings in 3-space must always be 
planar. And even more: the 
cyclic orientations $\sigma(e)$ at the edges of $C$ form -- when projected down to a link graph 
to rotators at the vertices of the link graph -- a rotation system at the link graph, see 
\autoref{fig:link_project}. 

   \begin{figure} [htpb]   
\begin{center}
   	  \includegraphics[height=3cm]{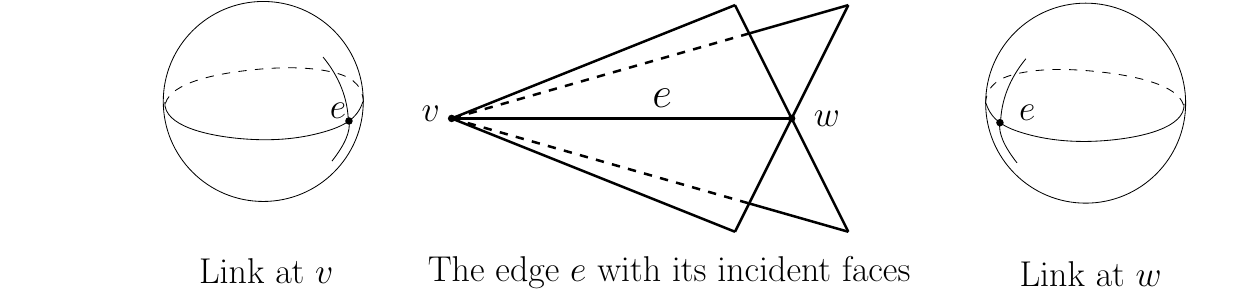}
   	  \caption{The cyclic orientation $\sigma(e)$ of the faces incident with the edge $e$ of 
$C$ projects down to rotators at $e$ in the link graphs at either endvertex of $e$. In these link 
graphs, the projected rotators at $e$ are reverse. 
}\label{fig:link_project}
\end{center}
   \end{figure}

Next we shall define `planar rotation systems' which roughly are rotation systems satisfying 
such an additional property.  
The cyclic orientation $\sigma(e)$ at the edge $e$ of a rotation system defines a rotation 
system $r(e,v, \Sigma)$ at each vertex $e$ of 
a link graph $L(v)$: if the directed edge $\vec{e}$ is directed towards $v$ we take $r(e,v, 
\Sigma)$ to be $\sigma(e)$. 
Otherwise we take the inverse of $\sigma(e)$. As explained in \autoref{basics}, this 
defines an embedding of the link graph into an oriented surface. The 
\emph{link complex} for 
$(C,\Sigma)$ at the vertex $v$ is the cell complex obtained from the link graph $L(v)$ 
by 
adding the faces of the above 
embedding of $L(v)$ into the oriented surface.
By definition,  the geometric realisation of the link complex is 
always a surface. To shortcut notation, we will not distinguish between the 
link complex and its geometric realisation and just say things like: 
`the link complex is a sphere'. 
A \emph{planar rotation system} of a directed simplicial complex $C$ is a rotation system such that 
for 
each vertex $v$ all link complexes are spheres -- or disjoint unions of spheres (if the link graph 
is not connected). 
The paragraph before shows the following. 

\begin{obs}\label{obo}
Rotation systems induced by topological embeddings  of locally connected\footnote{
\autoref{obo} is also true without the assumption of `local connectedness'. In that 
case however 
the link complex is 
disconnected. Hence it is no longer directly given by the drawing of the link graph on a 
ball of small radius as above.} simplicial complexes in the 3-sphere are planar. 
 \qed
\end{obs}

Next we will define the \emph{local surfaces of a topological embedding} of a simplicial complex 
$C$ into \Sthree. 
The local surface at a connected component of $\Sbb^3\sm C$ is the following. 
Pour concrete into this connected component. 
The surface of the concrete is a 2-dimensional manifold. 
The local surface is the simplicial complex drawn at the surface by the vertices, edges and faces 
of $C$.
Note that if an edge $e$ of $G$ is incident with more than two faces that are on the 
surface, then the surface will contain at least two clones of the edge $e$, see 
\autoref{fig:loc_surf}. 

   \begin{figure} [htpb]   
\begin{center}
   	  \includegraphics[height=3cm]{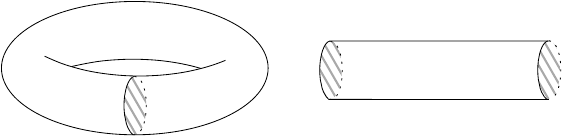}
   	  \caption{On the left we depicted the torus with an additional face 
attached on a genus reducing curve in the inside. On the right we depicted the 
local surface of its inside component. It is a sphere and contains two 
copies of the newly added face (and its incident edges).}\label{fig:loc_surf}
\end{center}
   \end{figure}

Now we will define \emph{local surfaces for a pair $(C,\Sigma)$} consisting of a complex 
$C$ and one of its rotation systems 
$\Sigma$. 
\autoref{topo_to_combi} below says that under fairly general circumstances the local surfaces of a 
topological embedding are the local 
surfaces of the rotation system induced by that topological embedding. The set of faces of 
a local surface will be an equivalence 
class of the set of orientations of faces of $C$.
The \emph{local-surface-equivalence relation} is the symmetric transitive closure of the following 
relation.
An orientation $\vec f$ of a face $f$ is \emph{locally related} via an edge $e$ of 
$C$ to an orientation $\vec{g}$ of a face $g$ if $f$ is just before $g$ in $\sigma(e)$ and 
$e$ 
is traversed positively by $\vec f$ and 
negatively by $\vec g$ and in $\sigma(e)$ the 
faces $f$ and $g$ are adjacent. 
Here we follow the convention that if the edge $e$ is only incident with a single face, then the 
two orientations of that face 
are related via $e$. 
Given an equivalence class of the local-surface-equivalence relation, the \emph{local surface} at 
that equivalence class is the following 
 complex whose set of faces is (in bijection with) the set of orientations 
in that equivalence class. 
We obtain the complex from the disjoint union of the faces of these orientations by 
gluing together two of these faces $f_1$ and $f_2$ along two of their edges if these edges are 
copies 
of the same 
edge $e$ of $C$ and 
$f_1$ and $f_2$ are related via $e$. Of course, we glue these two edges in a 
way that endvertices are identified only with copies of the same vertex of $C$.
Hence each  edge of a local surface is incident with precisely two faces. Hence 
its 
geometric realisation is always a 
is a surface. Similarly as for link complexes, we shall just say things like `the 
local surface is a sphere'. 
\begin{obs}\label{loc_is_con}
Local surfaces of planar rotation systems are always connected.
\qed
\end{obs}
A \emph{(2-dimensional) orientation} of a complex $C$ such that each edge is in precisely two faces 
is a choice of orientation of each 
face of $C$ such that each edge is traversed in opposite directions by the chosen orientation of 
the two incident faces.
Note that a complex whose geometric realisation is a surface has an orientation if and only its 
geometric realisation is orientable. 

\begin{obs}\label{loc_is_orientable}
The set of orientations in a local-surface-equivalence class defines an orientation of its local 
surface. 

In particular, local surfaces are cell complexes. 
\qed
\end{obs}

We will not use the following lemma in our proof of \autoref{combi_intro} 
 However, we think that it gives a 
better intuitive understanding of local surfaces.  We say that a simplicial complex $C$ is 
\emph{locally connected} if all link graphs are connected. 

\begin{lem}\label{topo_to_combi}
 Let $C$ be a connected and locally connected complex embedded into $\Sbb^3$ and let $\Sigma$ be 
the induced planar rotation system. 
 Then the local surfaces of the topological embedding are equal to the local surfaces for 
$(C,\Sigma)$.\qed
\end{lem}

There is the following relation between vertices of local surfaces and faces of link 
complexes.

\begin{lem}\label{loc_inc_AND_loc_surfaces}
Let $\Sigma$ be a rotation system of a simplicial complex $C$. 
There is a bijection $\iota$ between the set of vertices of local 
surfaces for $(C,\Sigma)$ and the set of faces of link 
complexes for $(C,\Sigma)$, which maps each vertex $v'$ of a local surface cloned from the vertex 
$v$ of $C$ to a face $f$ of the link complex at $v$ such that the rotation system at $v'$ is an 
orientation of $f$.
\end{lem}
\begin{proof}
The set of faces of the link complex at $v$ is in bijection with the set of 
$v$-equivalence classes; here the \emph{$v$-equivalence relation} on the set of orientations of 
faces of $C$ 
incident with $v$ is 
the symmetric transitive closure of the relation `locally related'. Since we work in a subset of 
the orientations, every $v$-equivalence 
class is contained in a local-surface-equivalence class. 
On the other hand the set of all clones of a vertex $v$ of $C$ contained in a local surface $S$ is 
in bijection with the set of 
$v$-equivalence classes contained in the local-surface-equivalence class of $S$. 
This defines a bijection $\iota$ between the set of vertices of local 
surfaces for $(C,\Sigma)$ and the set of faces of link 
complexes for $(C,\Sigma)$.

It is straightforward to check that $\iota$ has all the properties claimed in the lemma. 
\end{proof}

\begin{cor}\label{cor7}
Given a local surface of a simplicial complex $C$ and one of its vertices $v'$ cloned from a 
vertex 
$v$ of $C$, there is a homeomorphism from a neighbourhood around $v'$ in the local surface to the 
cone with top $v'$
over the face boundary of $\iota(v')$ that fixes $v'$ and the edges and faces incident with 
$v'$ in a neighbourhood around $v'$.
\qed
\end{cor}

The definitions of link graphs and link complexes can be generalised from 
simplicial complexes to 
complexes as follows. 
The \emph{link graph} of a complex $C$ at a vertex $v$ is the graph whose vertices are 
the edges incident with $v$.
For any traversal of a face of the vertex $v$, we add an edge between the two vertices that when 
considered as edges of $C$ 
are in the face just before and just after that traversal of $v$. We stress that we allow parallel 
edges 
and loops.
Given a complex $C$, any rotation system $\Sigma$ of $C$ defines rotation systems at each link 
graph of $C$.
Hence the definition of link complex extends.  

\section{Constructing piece-wise linear embeddings}\label{sec4}

In this section we prove \autoref{combi} below, which is used in the proof of 
\autoref{combi_intro}. 
   \begin{figure} [htpb]   
\begin{center}
   	  \includegraphics[height=2cm]{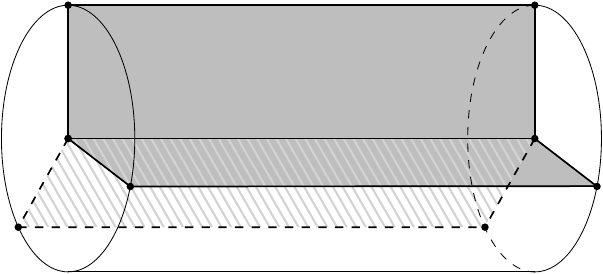}
   	  \caption{A cylinder with an embedded 2-complex.}\label{fig:cylinder}
\end{center}
   \end{figure}
\begin{eg}\label{solid_klein}
Here we give the definition of the Solid Klein Bottle and construct an embedding of a 2-complex $C$ 
in the Solid Klein Bottle that does not induce a planar rotation system. 

Given a solid cylinder (see \autoref{fig:cylinder}), there are two ways to identify the 
bounding discs: one way the boundary becomes a torus and the other way the boundary becomes a 
Klein Bottle. We refer to the topological space obtained from the solid cylinder by identifying as 
in the second case as the \emph{Solid Klein Bottle}. In \autoref{fig:cylinder} we embedded a 
2-complex in 
the solid cylinder. Extending the above identification of the cylinder to the embedded 2-complex, 
induces an embedding of this new 2-complex in the Solid Klein Bottle. This embedding does not 
induce 
a planar rotation system.
\end{eg}

Next we show that the Solid Klein Bottle is the only obstruction that prevents embeddings of 
2-complexes in 3-manifolds to induce planar rotation systems. 

\begin{lem}\label{prs_klein}
 Let $M$ be a 3-manifold that does not include the Solid Klein Bottle as a submanifold.
 Then any embedding of a 2-complex $C$ in $M$ induces a planar rotation system.
\end{lem}

\begin{proof}
 By treating different connected components separately, we may assume that $C$ is connected. 
 Around each vertex of $C$ we pick a small neighbourhood that is an open 3-ball, and around each 
edge of $C$ we pick a small neighbourhood that is an open cylinder. Next we define orientations on 
these neighbourhoods. For that we pick an arbitrary vertex and pick for its neighbourhood one of 
the two orientations arbitrarily. Then we pick compatible orientations at the 
neighbourhoods of the incident edges. As $C$ is connected, we can continue recursively along a 
spanning tree of $C$ until we picked orientations at all neighbourhoods of vertices and edges. If 
for a vertex and an incident edge, their neighbourhoods have incompatible orientations, this edge 
must be outside the spanning tree and we can build a Solid Klein Bottle from its fundamental cycle 
as follows. Indeed, sticking together the balls at the vertices and the cylinders at the edges gives 
a Solid Klein Bottle.
By assumption, this does not occur, so all orientations of neighbourhoods are compatible.

Let $\Sigma$ be the rotation system induced by this embedding of $C$ in $M$ with respect to the 
orientations chosen above. Clearly, this rotations system $\Sigma$ is planar. 
\end{proof}

\begin{cor}\label{prs_orient}
 Any embedding of a 2-complex $C$ in an orientable 3-manifold $M$ induces a planar rotation system.
\end{cor}

\begin{proof}
 It is well-known that a 3-manifold is orientable if and only if it does not have the Solid Klein 
Bottle as a submanifold. Hence \autoref{prs_orient} follows from \autoref{prs_klein}. 
\end{proof}

Throughout this section we fix a connected and locally connected simplicial complex $C$ with a 
rotation system $\Sigma$. 
An associated topological space $T(C,\Sigma)$ is defined as follows.
For each local surface $S$ of $(C,\Sigma)$ we take an embedding into \Sthree\ as follows. 
Let $g$ be the genus of $S$. We start with the unit ball in \Sthree\ and then identify $g$ disjoint 
pairs of discs through the outside.\footnote{We have some flexibility along which paths on the 
outside we do these identifications but we do not need to make particular choices for our 
construction to work.} The constructed surface is isomorphic to $S$, so this defines an embedding of 
$S$. 
Each local surface 
is oriented and we denote by $\hat S$ the topological space obtained from \Sthree\ by deleting all 
points on the outside of $S$. We obtain  $T(C,\Sigma)$ from the simplicial complex $C$ by gluing 
onto each 
local surface $S$ the topological space $\hat S$ along $S$. 

We remark that associated topological spaces may depend on the chosen embeddings of the local 
surfaces $S$ into \Sthree. However, if all local surfaces are spheres, then any two associated 
topological spaces are isomorphic and in this case we shall talk about `the' associated topological 
space. 

Clearly, associated topological spaces $T(C,\Sigma)$ are compact and connected as $C$ is connected.

\begin{lem}\label{is_manifold}
The rotation system $\Sigma$ is planar if and only if the associated topological 
space $T(C,\Sigma)$ is an oriented 3-dimensional manifold. 
\end{lem}

 \begin{proof}
 \autoref{prs_orient} implies that if the topological space $T(C,\Sigma)$ is a 3-dimensional 
orientable 3-manifold, then the rotation system $\Sigma$ is planar. 
 Conversely, now assume that $\Sigma$ is a planar rotation system. 
  We have to show that the topological space $T(C,\Sigma)$ is an orientable 3-manifold. 
  It suffices to show that $T(C,\Sigma)$ is a 3-manifold since then orientability follows 
immediately from the construction of $T(C,\Sigma)$. 
So we are to show that there is a neighbourhood around any point $x$ of  $T(C,\Sigma)$ that is 
isomorphic 
to the closed 3-dimensional ball $B_3$.

 If $x$ is a point not in $C$, this is clear. If $x$ is an interior point of a face $f$, we obtain 
a neighbourhood of $x$ by gluing together neighbourhoods of copies of $x$ in the local surfaces 
that 
contain an orientation of $f$.
Each orientation of $f$ is contained in local surfaces exactly once. Hence we glue together the two 
orientations of $f$ and
clearly $x$ has a neighbourhood isomorphic to $B_3$.

Next we assume that $x$ is an interior point of an edge $e$.  Some open neighbourhood of 
$x$ is isomorphic to the 
topological space obtained from gluing 
together for each copy of $e$ in a local surface, a neighbourhood around a copy $x'$ of $x$ on 
those edges.
A neighbourhood around $x'$ has the shape of a piece of a cake, see \autoref{fig:piece_of_cake}

   \begin{figure} [htpb]   
\begin{center}
   	  \includegraphics[height=2cm]{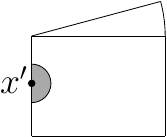}
   	  \caption{A piece of a cake. This space is obtained by taking the product of a triangle 
with the unit interval. The edge $e$ is mapped to the set of points 
corresponding to some vertex of the triangle. }\label{fig:piece_of_cake}
\end{center}
   \end{figure}

First we consider the case that $x$ has several copies.
As $\sigma(e)$ is a cyclic orientation, these pieces of a cake are glued together in a cyclic way 
along faces. Since each cyclic 
orientation of a face appears exactly once in local surfaces, we identify in each gluing step the 
two cyclic orientations of a face. 
Informally, the overall gluing will be a `cake' with $x$ as an interior point. 
Hence a neighbourhood of $x$ is isomorphic to $B_3$. 
If there is only one copy of $x'$, then the copy of $e$ containing $x'$ is incident with the two 
orientations of a single face.
Then we obtain a neighbourhood of $x$ by identifying these two orientations. Hence there is a 
neighbourhood of $x$ isomorphic to $B_3$.

It remains to consider the case where $x$ is a vertex of $C$. We obtain a neighbourhood of $x$ by 
gluing together neighbourhoods of copies 
of $x$ in local surfaces. We shall show that we have one such copy for every face of the link 
complex for $(C,\Sigma)$ and a neighbourhood of $x$ in such a copy is given by the cone over that 
face with $x$ being the top of the cone, see \autoref{fig:paste}. 
   \begin{figure} [htpb]   
\begin{center}
   	  \includegraphics[height=4cm]{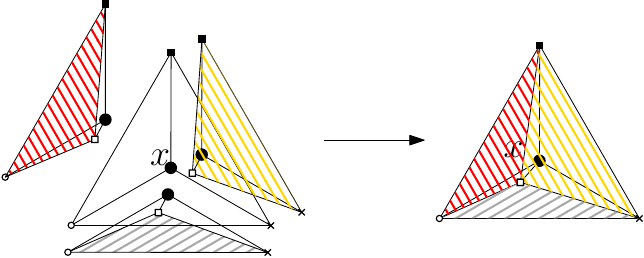}
   	  \caption{In this example the link complex of $x$ is a 
tetrahedron. The three faces visible in our drawing are highlighted in red, gold and grey.     
   	  On the left we see how the four cones over the faces of the link complex  
are pasted together to form the 
cone over the link complex depicted on the right. }\label{fig:paste}
\end{center}
   \end{figure}
We shall show that the glued together neighbourhood 
is the cone over the link complex with $x$ at the top. Since $\Sigma$ is planar and $C$ is locally 
connected, the link complex is isomorphic to the 2-sphere. Since the cone over the 2-sphere is a 
3-ball, the neighbourhood of $x$ has the desired type.

Now we examine this plan in detail. 
By \autoref{loc_inc_AND_loc_surfaces} and \autoref{cor7}, the copies are 
mapped by the 
bijection $\iota$ to the faces of the link complex at $x$ and a neighbourhood around such a copy 
$x'$ is isomorphic to the cone  
with top $x'$ over the face $\iota(x')$.
We glue these cones over the faces $\iota(x')$ on their faces that are obtained from edges of 
$\iota(x')$ by adding the top $x'$. 

The glued together complex is isomorphic to the cone over the complex $S$ obtained by gluing 
together the faces $\iota(x')$ along edges, where we always glue the edge the way round so that 
copies of the same vertex of the local 
incidence graph are identified. Hence the vertex-edge-incidence relation and the 
edge-face-incidence relation of $S$ are the same as for 
the link complex at $x$. The same is true for the cyclic orderings of edges on faces.
So $S$ is equal to the link complex at $x$.

Hence a neighbourhood of $x$ is isomorphic to a cone with top $x$ over the link complex 
at $x$. Since $\Sigma$ is a 
planar rotation system, the link complex is a disjoint union of spheres. As $C$ is 
locally connected, it is a sphere. Thus its 
cone is isomorphic to $B_3$. 
 \end{proof}

\begin{lem}\label{is simply connected}
 If $C$ is simply connected, then any associated topological space $T(C,\Sigma)$ is simply 
connected.
\end{lem}
\begin{proof}
This is a consequence of Van Kampen's Theorem {\cite[Theorem 1.20]{{Hatcher}}}. 
Indeed, we obtain $T$ 
from $T(C,\Sigma)$ by deleting all interior points of the sets $\hat S$ for local surfaces $S$ that 
are not in a small open neighbourhood of $C$. This can be done in such a 
way that $T$ has a deformation retract to $C$, and thus is 
simply connected. 
Now we recursively glue the spaces $\hat S$ back onto $T$. 
In each step we glue a single space $\hat S$. Call the space obtained after $n$ gluings $T_n$.

The fundamental group of $\hat S$ is a quotient of the fundamental group of the 
intersection of $T_n$ and $\hat S$ (this follows from the construction of the embedding of the 
surface $S$ into the space $\hat S$). And the fundamental group of $T_n$ is trivial by induction. 
So we can apply  
Van Kampen's Theorem to deduce that the gluing space $T_{n+1}$ has trivial fundamental group. Hence 
the final gluing space $T(C,\Sigma)$ has trivial fundamental group. So it is simply connected. 
\end{proof}

The converse of \autoref{is simply connected} is true if all local surfaces for $(C,\Sigma)$ are 
spheres as follows. 

\begin{lem}\label{is_simply_connected2}
 If all local surfaces for $(C,\Sigma)$ are spheres and the associated topological space 
$T(C,\Sigma)$ is simply connected, then $C$ is simply connected.
\end{lem}

\begin{proof}
Let $\varphi$ an image of $\Sbb^1$ in $C$. Since  $T(C,\Sigma)$ is simply connected, there is 
a homotopy from $\varphi$ to a point of $C$ in $T(C,\Sigma)$.
We can change the homotopy so that it avoids an interior point of each local 
surface of the embedding. 
Since each local surface is a sphere, for each local surface without the chosen point there is a 
continuous projection to its 
boundary. Since these projections are continuous, the concatenation of them with 
the homotopy is continuous. Since this concatenation is constant on $C$ this 
defines a homotopy of $\varphi$ inside $C$. Hence $C$ is simply connected.   
\end{proof}

We conclude this section with the following special case of \autoref{combi_intro}. 

\begin{thm}\label{combi}
Let $\Sigma$ be a rotation system of a locally connected simplicial complex $C$.
Then $\Sigma$ is planar if and 
only if $T(C,\Sigma)$ is an oriented 3-manifold. 
And if $\Sigma$ is planar and $C$ is simply connected, then $T(C,\Sigma)$ must be the 3-sphere. 
\end{thm}

\begin{proof}
By treating different connected components separately, we may assume that $C$ is connected.
The first part follows from \autoref{is_manifold}. The second part follows from \autoref{is simply 
connected} 
and Perelman's theorem \cite{Perelman1, Perelman2, Perelman3} that any compact simply connected 
3-manifold is isomorphic to the 3-sphere.
\end{proof}

\begin{rem}\label{P_rem}
We used Perelman's theorem in the proof of \autoref{combi}. On the other hand it 
together with Moise's theorem \cite{moise}, which says 
that every compact 3-dimensional manifold has a triangulation, implies Perelman's theorem as 
follows. Let $M$ be a simply connected 3-dimensional compact manifold. 
 Let $T$ be a triangulation of $M$. And let $C$ be the simplicial complex obtained from $T$ by 
deleting the 3-dimensional cells. 
Let $\Sigma$ be the rotation system given by the embedding of $C$ into $T$. 
It is clear from that construction that $T$ is equal to the triangulation given by 
the embedding of $C$ into $T(C,\Sigma)$. Hence we can apply \autoref{is_simply_connected2} to 
deduce that $C$ is simply connected. 
Hence by \autoref{combi} the topological space $T(C,\Sigma)$, into which $C$ embeds, is isomorphic 
to the 3-sphere. Since $T(C,\Sigma)$ is isomorphic to $M$, we deduce that $M$ is isomorphic to  the 
3-sphere. 
\end{rem}

\section{Cut vertices}\label{sec5}

In this section we deduce \autoref{combi_intro} from \autoref{combi} proved in the last section.
Given a prime $p$, a \scom\ $C$ is \emph{$p$-nullhomologous} if every directed cycle of $C$ is 
generated over $\Fbb_p$ by the boundaries of faces of $C$. Note that a simplicial complex $C$ 
is  $p$-nullhomologous if and only if the first homology group $H_1(C,\Fbb_p)$ is trivial.
Clearly, every simply connected \scom\ is $p$-nullhomologous.

A vertex $v$ in a connected complex $C$ is a \emph{cut vertex} if the 1-skeleton of $C$ without $v$ 
is a disconnected graph\footnote{We 
define this in terms of the 1-skeleton instead of directly in terms of $C$ for a technical reason: 
The object obtained from a simplicial 
complex by deleting a vertex may have edges not incident with faces. So it would not be a 
2-dimensional simplicial complex in the 
terminology of this paper.}.
A vertex $v$ in an arbitrary, not necessarily connected, complex $C$ is \emph{a cut vertex} if it 
is a 
cut vertex in a connected 
component of $C$.

\begin{lem}\label{is_loc_con}
 Every $p$-nullhomologous simplicial complex without a cut vertex is locally connected. 
\end{lem}

\begin{proof}
 We construct for any vertex $v$ of an arbitrary simplicial complex $C$ such that the link graph 
$L(v)$ at $v$ is not 
connected and $v$ is not a cut vertex a cycle containing $v$ that is not generated by the face 
boundaries of $C$. 

Let $e$ and $g$ be two vertices in different components of $L(v)$. These are edges of $C$ and let 
$w$ and $u$ be their endvertices 
different from $v$. Since $v$ is not a cut vertex, there is a path in $C$ between $u$ and $w$ that 
avoids $v$.
This path together with the edges $e$ and $g$ is a cycle $o$ in $C$ that contains $v$.

Our aim is to show that $o$ is not generated by the boundaries of faces of $C$. 
Suppose for a contradiction that $o$ is generated. Let $F$ be a family of faces whose boundaries 
sum 
up to $o$.
Let $F_v$ be the subfamily of faces of $F$ that are incident with $v$. Each face in $F_v$ is an 
edge of $L(v)$ and each vertex of $L(v)$ 
is incident with an even number (counted with multiplicities) of these edges except for $e$ and $g$ 
that are incident with an odd number of 
these faces. 
Let $X$ be the connected component of the graph $L(v)$ restricted to the edge set $F_v$ that 
contains the vertex $e$. 
We obtain $X'$ from $X$ by adding $k-1$ parallel edges to each edge that appears $k$ times in 
$F_v$. 
Since $X'$ has an even number of vertices of odd degree also $g$ must be in $X$. This is a 
contradiction to the assumption that $e$ and $g$ are in different components of $L(v)$. Hence $o$ 
is not generated by the boundaries of 
faces of $C$. This completes the proof. 
\end{proof}

Given a connected complex  $C$ with a cut vertex $v$ and a connected component $K$ of the 
1-skeleton of $C$ with $v$ deleted, the 
\emph{complex attached} at $v$ 
centered at $K$ has vertex set $K+v$ and its edges and faces are those of $C$ all of whose incident 
vertices are in $K+v$.

\begin{lem}\label{block-lem}
A connected simplicial complex $C$ with a cut vertex $v$ has a piece-wise linear 
embedding into $\Sbb^3$ 
if and only if all complexes attached at 
$v$ have 
a piece-wise linear embedding into $\Sbb^3$. 
\end{lem}

\begin{proof}
 If $C$ has an embedding into $\Sbb^3$, then clearly all complexes attached at $v$ have an 
embedding. Conversely suppose that all 
complexes attached at $v$ have an embedding into $\Sbb^3$. Pick one of these complexes arbitrarily, 
call it $X$ and fix an embedding of it 
into $\Sbb^3$. In that embedding pick for each component of $C$ remove $v$ except that for $X$ a 
closed ball contained in $\Sbb^3$ that 
intersects $X$ precisely in $v$ such that all these closed balls intersect pairwise only at $v$. 
Each complex attached at $v$, has a 
piece-wise linear embedding into the 
3-dimensional unit ball as they have embeddings into $\Sbb^3$ such that some open set is disjoint 
from the complex. Now we attach these 
embeddings into the balls of the embedding of $X$ inside the reserved balls by identifying the 
copies of $v$. This defines an embedding of 
$C$.
\end{proof}

Recall that in order to prove \autoref{combi_intro} it suffices to show that any simply 
connected simplicial complex $C$ has a piece-wise linear embedding into 
$\Sbb^3$ if and only if $C$ has a planar rotation system. 

\begin{proof}[Proof of \autoref{combi_intro}.]
Clearly if a simplicial complex is embeddable into $\Sbb^3$, then it has a planar rotation system. 
For the other implication, let $C$ be a simply connected simplicial complex and $\Sigma$ be a 
planar rotation system. 
We prove the theorem by induction on the number of cut vertices of $C$. If $C$ has no cut vertex, 
it is locally connected by \autoref{is_loc_con}. Thus it has a piece-wise linear embedding into 
\Sthree\ by \autoref{combi}. 

Hence we may assume that $C$ has a cut vertex $v$. 
As $C$ is simply connected, every complex attached at $v$ is simply connected. Hence by the 
induction hypothesis each of these complexes has a piece-wise linear embedding into \Sthree. 
Thus $C$ has a piece-wise linear embedding into \Sthree\ by  \autoref{block-lem}.
\end{proof}

\section{Local surfaces of planar rotation systems}\label{sec6}

The aim of this section is to prove \autoref{combi_intro_extended}. A shorter proof is sketched 
in \autoref{alg_topo} using algebraic topology. As a first step in that 
direction, we first prove the following.

\begin{thm}\label{loc_are_spheres}
 Let $C$ be a locally connected $p$-nullhomologous \scom\ that has a planar rotation system. Then 
all 
local 
surfaces of the planar rotation system are spheres. 
\end{thm}

Before we can prove \autoref{loc_are_spheres} we need some preparation.
The complex \emph{dual} to a simplicial $C$ with a rotation system $\Sigma$ has as its set of 
vertices the set of local 
surfaces of 
$\Sigma$. 
Its set of edges is the set of faces of $C$, and an edge is incident with a vertex if the 
corresponding face is in the corresponding local 
surface. The faces of the dual are the edges of $C$. Their cyclic ordering is as given by $\Sigma$. 
In particular, the 
edge-face-incidence-relation of the dual is the same as that of $C$ but with the roles of edges and 
faces interchanged. 

Moreover, an orientation $\vec{f}$ of a face $f$ of $C$ corresponds to the direction of $f$ when 
considered as an edge of the dual complex  
$D$ that points 
towards 
the vertex of $D$ whose local-surface-equivalence class contains $\vec{f}$. Hence the direction of 
the dual complex $C$ induces a direction 
of 
the complex $D$. 
By $\Sigma_C=(\sigma_C(f)|f\in E(D))$ we denote the following rotation system for $D$: for 
$\sigma_C(f)$ we take the orientation 
$\vec{f}$ of $f$ 
in the directed complex $C$. 

In this paper we follow the convention that for edges of $C$ we use the letter $e$ (with possibly 
some subscripts) while for 
faces of $C$ we use the letter $f$. In return, we use the letter $f$ for the edges of a dual 
complex of $C$ and $e$ for its faces.

\begin{lem}\label{D_con}
 Let $C$ be a connected and locally connected simplicial complex. Then for any rotation system, 
the dual complex $D$ is connected.
 \end{lem}

\begin{proof}
 Two edges of $C$ are \emph{$C$-related} if there is a face of $C$ incident with both of them. 
 And they are \emph{$C$-equivalent} if they are in the transitive closure of the symmetric relation 
`$C$-related'.
 Clearly, any two $C$-equivalent edges of $C$ are in the same connected component. If $C$ however 
is locally connected, also the converse 
is true: any two edges in the same connected component are $C$-equivalent. Indeed, take a path 
containing these two edges. Any 
two edges incident with a common vertex are $C$-equivalent as $C$ is locally connected. Hence any 
two edges on the path are $C$-equivalent.

We define \emph{$D$-equivalent} like `$C$-equivalent' with `$D$' in place of `$C$'. 
Now let $f$ and $f'$ be two edges of $D$. Let $e$ and $e'$ be edges of $C$ incident with $f$ and 
$f'$, respectively. Since $C$ is connected 
and locally connected the edges $e$ and $e'$ are $C$-equivalent. As $C$ and $D$ have the same 
edge/face incidence relation, the edges $f$ 
and $f'$ of $D$ are $D$-equivalent. So any two edges of $D$ are $D$-equivalent. Hence $D$ is 
connected.
\end{proof}

First, we prove the following, which is reminiscent of euler's formula. 
 
 \begin{lem}\label{geq}
Let $C$ be a locally connected $p$-nullhomologous \scom\ with a planar rotation system and $D$ the 
dual complex. 
Then 
\[
 |V(C)|-|E|+|F|-|V(D)|\geq 0
\]

Moreover, we have equality if and only if $D$ is $p$-nullhomologous.
\end{lem}

\begin{proof}
Let $Z_C$ be the dimension over $\Fbb_p$ of the cycle 
space of $C$. Similarly we define $Z_D$. 
Let $r$ be the rank of the edge-face-incidence matrix over $\Fbb_p$. Note that $r\leq Z_D$ and that 
$r=Z_C$ as $H_1(C,\Fbb_p)=0$. 
So $Z_D-Z_C\geq 0$. Hence it suffices to prove the following. 
 \begin{sublem}\label{euler_cycle_space}
\[ |V(C)|-|E|+|F|-|V(D)|= Z_D-Z_C \]
 \end{sublem}
 \begin{proof}
 Let $k_C$ be the number of connected components of $C$ and $k_D$ be the number of connected 
components of $D$. 
Recall that the space orthogonal to the cycle space (over $\Fbb_p$) in a graph $G$ has dimension 
$|V(G)|$ minus the number of connected 
components of $G$. Hence $Z_C=|E|-|V(C)|+k_C$ and $Z_D=|F|-|V(D)|+k_D$. Subtracting the first 
equation from the second yields:
\[
|V(C)|-|E|+|F|-|V(D)|+(k_D-k_C)= Z_D-Z_C 
\]
 Since the dual complex of the disjoint union of two simplicial complexes (with planar rotation 
systems) is the disjoint union of their 
dual 
complexes, $k_C\leq k_D$. By \autoref{D_con} $k_C= k_D$. Plugging this into the equation before, 
proves the sublemma.
 \end{proof}
 
 This completes the proof of the inequality. We have equality if and only if $r= Z_D$. So the 
`Moreover'-part follows.
 \end{proof}

Our next goal is to prove the following, which is also reminiscent of euler's formula but here the 
inequality goes the other way round.

\begin{lem}\label{euler_double_counting}
 Let $C$ be a locally connected \scom\ with a planar rotation system $\Sigma$ and $D$ the dual 
complex.
 Then: \[
 |V(C)|-|E|+|F|-|V(D)|\leq 0 \] with equality if and only if all link complexes for 
$(D,\Sigma_C)$
are spheres. 
\end{lem}

Before we can prove this, we need some preparation. 
 By $a$ we denote the sum of the faces of link complexes for 
$(C,\Sigma)$. By $a'$ we denote the sum over the faces of link complexes for 
$(D,\Sigma_C)$.
  Before proving that $a$ is equal to $a'$ we prove that it is useful by showing the following. 
 
 \begin{claim}\label{claim1}
  \autoref{euler_double_counting} is true if $a=a'$ and all link complexes for 
$(D,\Sigma_C)$ are connected.
 \end{claim}

\begin{proof}
Given a face $f$ of $C$, we denote the number of edges incident with $f$ by 
$deg(f)$. 
Our first aim is to prove that

\begin{equation}\label{eq1}
  2 |V(C)|= 2 |E|-\sum_{f\in F} \degree(f) + a
\end{equation}

To prove this equation, we apply Euler's formula \cite{DiestelBookCurrent} in the link complexes 
for $(C,\Sigma)$. 
Then we take the sum of all these equations over all $v\in V(C)$.
Since $\Sigma$ is a planar rotation system, all link complexes are a disjoint union of 
spheres.
Since $C$ is locally connected, all link complexes are connected and hence are spheres.
So they have euler characteristic two. Thus we get the term $2 |V(C)|$ on 
the left hand side.
By definition, $a$ is the sum of the faces of link complexes for 
$(C,\Sigma)$.

The term $2 |E|$ is the sum over all vertices of link complexes for 
$(C,\Sigma)$.
Indeed, each edge of $C$ between the two vertices $v$ and $w$ of $C$ is a vertex 
of precisely the two link complexes for $v$ and $w$. 

The term $\sum_{f\in F} \degree(f)$ is the sum over all edges of link complexes for $(C,\Sigma)$.
Indeed, each face $f$ of $C$ is in precisely those link complexes for 
vertices on the boundary of $f$.
This completes the proof of \eqref{eq1}.

Secondly, we prove the following inequality using a similar argument.
Given an edge $e$ of $C$, we denote the number of faces incident with $e$ by 
$\degree(e)$. 
\begin{equation}\label{eq2}
  2 |V(D)| \geq 2 |F|-\sum_{e\in E} \degree(e) + a'
\end{equation}

To prove this, we apply Euler's formula in link complexes 
for $(D,\Sigma_C)$, and take the sum over all $v\in V(D)$.
Here we have `$\geq$' instead of `$=$' as we just know by assumption that the link 
complexes are connected but they may not be a sphere.
So we have $2 |V(D)|$ on the left and $a'$ is the sum over the faces of link complexes for 
$(D,\Sigma_C)$.

The term $2 |F|$ is the sum over all vertices of link complexes for 
$(D,\Sigma_C)$.
Indeed, each edge of $D$ between the two different vertices $v$ and $w$ of $D$ is a vertex 
of precisely the two link complexes for $v$ and $w$. A loop gives rise to two vertices 
in the link graph at the vertex 
it is attached to. 

The term $\sum_{e\in E} \degree(e)$ is the sum over all edges of link complexes for $(D,\Sigma_C)$.
Indeed, each face $e$ of $D$ is in the link complex at $v$ with multiplicity equal to 
the number of times it traverses $v$. 
This completes the proof of \eqref{eq2}.

By assumption, $a=a'$. 
The sums  $\sum_{f\in F} \degree(f)$ and $\sum_{e\in E} \degree(e)$ both count 
the number
of nonzero entries of $A$, so they are equal.
Subtracting \eqref{eq2} from \eqref{eq1}, rearranging and dividing by 2 yields:

\[
 |V(C)|-|E|+|F|-|V(D)|\leq 0
\]
 with equality if and only if all link complexes for $(D,\Sigma_C)$
are spheres.
\end{proof}
 
 \vspace{0.3 cm}
 
 Hence our next aim is to prove that $a$ is equal to $a'$. 
 First we need some preparation.

 Two cell complexes $C$ and $D$ are \emph{(abstract) 
surface 
duals} 
 if the set of vertices of $C$ is (in bijection with) the set of faces of $D$, 
 the set of edges of $C$ is the set of edges of $D$ and the set of faces of $C$ is the set of 
vertices of $D$.
 And these three bijections preserve incidences.
 
 \begin{lem}\label{loc_surface_is_inc}
 Let $C$ be a simplicial complex and $\Sigma$ be a rotation system and let $D$ be the dual. 
  The surface dual of a local surface $S$ for $(C,\Sigma)$ is equal to the link complex 
for $(D,\Sigma_C)$ at the vertex $\ell$ 
of $D$ that corresponds to $S$. 
 \end{lem}

 \begin{proof}
 It is immediate from the definitions that the vertices of the link complex $\bar L$ at $\ell$ 
are the faces of $S$. 
 The edges of $S$ are triples $(e,\vec{f},\vec{g})$, where $e$ is an edge of $C$ and $\vec{f}$ and 
$\vec{g}$ are orientations of faces of 
$C$ that are related via $e$ and are in the local-surface-equivalence class for $S$.
Hence in $D$, these are triples $(e,\vec{f},\vec{g})$ such that $\vec{f}$ and $\vec{g}$ are 
directions of edges that 
point towards $\ell$ and $f$ and $g$ are adjacent in the cyclic ordering of the face $e$. This are 
precisely the edges of the link graph 
$L(\ell)$. 
Hence the link graph $L(\ell)$ is the dual graph\footnote{ The \emph{dual graph} of a cell complex 
$C$ is the graph $G$
 whose set of vertices is (in bijection with) the set of faces of $C$ and whose set of edges is the 
set of edges of $C$. 
 And the incidence relation between the vertices and edges of $G$ is the same as the incidence 
relation between the faces and edges 
of $C$.} of the cell complex $S$. 

Now we will use the Edmonds-Hefter-Ringel 
rotation principle, see {\cite[Theorem 3.2.4]{{MoharThomassen}}}, to deduce that the link complex 
$\bar L$ at $\ell$ is the surface dual of $S$. We denote the unique cell complex that is a surface 
dual of $S$ by $S^*$. Above we have shown that $\bar L$ and $S^*$ have the same 1-skeleton. 
Moreover, the rotation systems at the vertices of the link 
complex $\bar L$ are given by the cyclic orientations 
in  the local-surface-equivalence class for $S$. By \autoref{loc_is_orientable} these 
local-surface-equivalence classes define an orientation of $S$. So $\bar L$ and $S^*$ have the same 
rotation systems. Hence by the Edmonds-Hefter-Ringel 
rotation principle $\bar L$ and $S^*$ have to be isomorphic. So $\bar L$ is a surface dual of $S$.
 \end{proof}

\begin{proof}[Proof of \autoref{euler_double_counting}.]
 Let $C$ be a locally connected simplicial complex and $\Sigma$ be a rotation system and let $D$ 
be the dual. Let $\Sigma_C$ be as 
defined above.
By \autoref{loc_is_con} and \autoref{loc_surface_is_inc} every link complex for 
$(D,\Sigma_C)$ is connected. 
By \autoref{claim1}, it suffices to show that the sum over all faces of link complexes 
of $C$ with respect to $\Sigma$ is equal 
to 
the sum over all faces of link complexes for $D$ with respect to $\Sigma_C$.
By \autoref{loc_surface_is_inc}, the second sum is equal to the sum over all vertices of local 
surfaces for $(C,\Sigma)$.
This completes the proof by \autoref{loc_inc_AND_loc_surfaces}. 
\end{proof}

 \begin{proof}[Proof of \autoref{loc_are_spheres}.]
 Let $C$ be a $p$-nullhomologous locally connected simplicial complex that has a planar rotation 
system $\Sigma$.
 Let $D$ be the dual complex.
 Then by \autoref{euler_double_counting} and \autoref{geq}, $C$ and $D$ satisfy Euler's formula, 
that is:
 \[
 |V(C)|-|E|+|F|-|V(D)|= 0
\]
 Hence by \autoref{euler_double_counting}  all link complexes for $(D,\Sigma_C)$ are 
spheres. By \autoref{loc_surface_is_inc} 
these are dual to the local surfaces for $(C,\Sigma)$. Hence all local surfaces for $(C,\Sigma)$ 
are spheres. 
 \end{proof}

The following theorem gives three equivalent characterisations of the class of locally 
connected simply connected simplicial complexes embeddable in   
 $\Sbb^3$.
 
\begin{thm}\label{nullt}
Let $C$ be a locally connected simplicial complex embedded into $\Sbb^3$. The following are 
equivalent.
\begin{enumerate}
 \item $C$ is simply connected; 
\item $C$ is $p$-nullhomologous for some prime $p$; 
\item all local surfaces of the planar rotation system induced by the topological embedding are 
spheres.
\end{enumerate}
\end{thm}

\begin{proof}
Clearly, 1 implies 2. To see that 2 implies 3, we assume that $C$ is $p$-nullhomologous.
Let $\Sigma$ be the planar rotation system induced by the topological embedding of $C$ into 
$\Sbb^3$. 
By \autoref{loc_are_spheres} all local 
surfaces for 
$(C,\Sigma)$ are spheres. 

It remains to prove that 3 implies 1.
So assume that $C$ has an embedding into 
$\Sbb^3$ such that all local surfaces of the planar rotation system induced by the topological 
embedding are spheres. 
By treating different connected components separately, we may assume that $C$ is connected. By 
\autoref{topo_to_combi} all local surfaces 
of the topological embedding are spheres. 
Thus 3 implies 1 by \autoref{is_simply_connected2}.

\end{proof}

\begin{rem}
Our proof actually proves the strengthening of \autoref{nullt} with `embedded 
into $\Sbb^3$' replaced by  `embedded into a simply connected 3-dimensional compact 
manifold.' However this strengthening is equivalent to \autoref{nullt} by Perelman's theorem. 
\end{rem}

Recall that in order to prove \autoref{combi_intro_extended}, it suffices to show that every 
$p$-nullhomologous simplicial complex $C$ has a piece-wise linear embedding into 
$\Sbb^3$ if and only if it is simply connected and $C$ has a planar rotation system.

\begin{proof}[Proof of \autoref{combi_intro_extended}.]
Using an induction argument on the number of cut vertices as in the proof of \autoref{combi_intro}, 
we may assume that $C$ is locally connected. 
If $C$ has a piece-wise linear embedding into 
$\Sbb^3$, then it has a planar rotation system and it is simply connected by \autoref{nullt}.
The other direction follows from \autoref{combi_intro}.
\end{proof}

\begin{rem}\label{alg_topo}
One step in proving \autoref{combi_intro_extended} was showing that if a simplicial 
complex whose first homology group is trivial embeds in \Sthree, then it must be simply connected. 
In this section we have given a proof that only uses elementary topology. We use these methods 
again in \cite{3space4}. 

However there is a shorter proof of this fact, which we shall sketch in the following. 
Let $C$ be a simplicial complex embedded in \Sthree\ such that one local surface of the embedding 
is not a sphere. Our aim is to show that the first homology group of $C$ cannot be trivial. 

We will rely on the fact that the first homology group of $X=\Sbb^3\sm \Sbb^1$ is not trivial. 
It suffices to show that the homology group of $X$ is a quotient of the homology group of $C$. 
Since here by Hurewicz's theorem, the homology group is the abelisation of the fundamental 
group, it suffices to show that the fundamental group $\pi_1(X)$ of $X$ is a quotient of the 
fundamental group $\pi_1(C)$. 

We let $C_1$ be a small open neighbourhood of $C$ in the embedding of $C$ in \Sthree. Since $C_1$ 
has a deformation retract onto $C$, it has the same fundamental group. 
We obtain $C_2$ from $C_1$ by attaching the interiors of all local surfaces of the embedding 
except for one -- which is not a sphere. This can be done by attaching finitely many 3-balls. 
Similar as in the proof of \autoref{is simply connected}, one can use Van Kampen's theorem to show 
that the fundamental group of $C_2$ is a quotient of the fundamental group of $C_1$. By adding 
finitely many spheres if necessary and arguing as above one may assume that remaining local surface 
is a torus. Hence $C_2$ has the same fundamental group as $X$. This completes the sketch. 
\end{rem}

\section{Embedding general simplicial complexes}\label{beyond}

There are three classes of simplicial complexes that naturally include the simply connected 
simplicial complexes:
the $p$-nullhomologous ones that are included in those with abelian fundamental group that in turn 
are included in 
general simplicial complexes. \autoref{combi_intro_extended} characterises embeddability of 
$p$-nullhomologous complexes. In this section we prove embedding results for the later two 
classes. The 
bigger the class gets, the stronger 
assumptions we will require in order to guarantee topological embeddings into $\Sbb^3$.

A \emph{curve system} of a surface $S$ of genus $g$ is a choice of at most $g$ genus reducing 
curves in $S$ that are disjoint. 
An \emph{extension} of a rotation system $\Sigma$ is a choice of curve system at every local 
surface of $\Sigma$.
An extension of a rotation system of a complex $C$ is \emph{simply connected} if the topological 
space obtained from $C$ by 
gluing\footnote{We 
stress that the curves need not go through edges of $C$. `Gluing' here is on the level of 
topological spaces not of complexes.} a disc at 
each curve of the extension is simply connected. 
The definition of a \emph{$p$-nullhomologous extension} is the same with `$p$-nullhomologous' in 
place of 
`simply connected'. 

\begin{thm}\label{general}
 Let $C$ be a connected and locally connected simplicial complex with a rotation system 
$\Sigma$. 
 The following are equivalent.
 \begin{enumerate}
  \item $\Sigma$ is induced by a topological embedding of $C$ into $\Sbb^3$.
  \item $\Sigma$ is a planar rotation system that has a simply connected extension. 
  \item We can subdivide edges of $C$, do baricentric subdivision of faces and add new faces such 
that the resulting simplicial complex is 
simply connected and has a topological embedding into \Sthree\ whose induced planar rotation system 
$\Sigma'$ `induces' $\Sigma$.
 \end{enumerate}
\end{thm}
Here we define that `$\Sigma'$ induces $\Sigma$' in the obvious way as follows. 
Let $C$ be a simplicial complex obtained from a simplicial complex $C'$ by deleting faces. 
A rotation system $\Sigma=(\sigma(e)|e\in E(C))$ of $C$ is \emph{induced} by a rotation system 
$\Sigma'=(\sigma'(e)|e\in E(C))$ of 
$C'$ if $\sigma(e)$ is the restriction of $\sigma'(e)$ to the faces incident with $e$. 
If $C$ is obtained from contracting edges of $C'$ instead, a rotation system $\Sigma$ of $C$ is 
\emph{induced} by a rotation system 
$\Sigma'$ of 
$C'$ if $\Sigma$ is the restriction of $\Sigma'$ to those edges that are in $C$. 
If $C'$ is obtained from $C$ by a baricentric subdivision of a face $f$ we take the same definition 
of `induced', where we make the 
identification between the face $f$ of $C$ and all faces of $C'$ obtained by subdividing $f$.
Now in the situation of \autoref{general}, we say that $\Sigma'$ \emph{induces} $\Sigma$ if there 
is a chain of planar rotation systems 
each inducing the next one starting with $\Sigma'$ and ending with $\Sigma$.  

Before we can prove \autoref{general}, we need some preparation. 
The following is a consequence of the Loop Theorem \cite{{Pap57},{Hatcher3notes}}.

\begin{lem}\label{cut_along_disc}
 Let $X$ be an orientable surface of genus $g\geq 1$ embedded topologically into $\Rbb^3$, then 
there is a genus reducing 
circle\footnote{A \emph{circle} is a topological space homeomorphic to $\mathbb{S}^1$. } $\gamma$ 
 of $X$ and a disc $D$ with boundary $\gamma$ and all interior points of $D$ are contained in the 
interior of $X$. 
\end{lem}

\begin{cor}\label{cut_along_discs}
  Let $X$ be an orientable surface of genus $g\geq 1$ embedded topologically into $\Rbb^3$, then 
there are genus reducing circles 
$\gamma_1$,..., $\gamma_g$ of $X$ and closed discs $D_i$ with boundary $\gamma_i$ 
such that the $D_i$  are disjoint and the interior points of the discs $D_i$ are contained in the 
interior of $X$.
\end{cor}

\begin{proof}
 We prove this by induction on $g$. In the induction step we cut of the current surface along $D$. 
Then we the apply 
\autoref{cut_along_disc} to that new surface. 
\end{proof}

\begin{proof}[Proof of \autoref{general}.]
1 is immediately implied by 3. 

Next assume that $\Sigma$ is induced by a topological embedding of $C$ into $\Sbb^3$. Then $\Sigma$ 
is clearly a planar rotation system. 
It has a simply connected extension by \autoref{cut_along_discs}. Hence 1 implies 2. 

Next assume that $\Sigma$ is a planar rotation system that has a simply connected extension. We 
can clearly subdivide edges and do 
baricentric subdivision and change the curves of the curve system of the simply connected extension 
such that in the resulting simplicial 
complex $C'$ all the curves of the simply connected extension closed are walks in the 1-skeleton of 
$C'$.
We define a planar rotation system $\Sigma'$ of $C'$ that induces $\Sigma$ as follows.
If we subdivide an edge, we assign to both copies the cyclic orientation of the original edge. If 
we do a baricentric subdivision, we 
assign to all new edges the unique cyclic orientation of size two. 
Iterating this during the construction of $C'$ defines $\Sigma'=(\sigma'(e)|e\in E(C))$, which 
clearly is a planar rotation system that 
induces $\Sigma$.
By construction $\Sigma'$ has a simply connected extension such that all its curves are walks in 
the 1-skeleton of $C'$ .

Informally, we obtain $C''$ from $C'$ by attaching a disc at the boundary of each curve of the 
simply connected extension. 
Formally, we obtain $C''$ from $C'$ by first adding a face for each curve $\gamma$ in the simply 
connected extension whose boundary is 
the closed walk $\gamma$. Then we do a baricentric subdivision to all these newly added faces. This 
ensures that 
$C''$ is a simplicial complex. 
Since $C$ is locally connected, also $C''$ is locally connected. 
Since the geometric realisation of $C''$ is equal to the geometric realisation of $C$, which is 
simply connected, the simplicial complex $C''$ 
is simply connected. 

Each newly added face $f$ corresponds to a traversal of a curve $\gamma$ of some edge $e$ of $C'$. 
This 
traversal is a unique edge of the local surface $S$ to whose curve system $\gamma$ belongs. For 
later reference we denote that copy of $e$ 
in $S$ by $e_f$.

We define a rotation system $\Sigma''=(\sigma''(e)|e\in E(C))$ of $C''$ as follows. 
All edges of $C''$ that are not edges of $C'$ are incident with precisely two faces. We take the 
unique cyclic ordering of size two there. 

Next we define $\sigma''(e)$ at edges $e$ of $C'$ that are incident with newly added faces. 
If $e$ is only incident with a single face of $C'$, 
then $e$ is only in a single 
local surface and it only has one copy in that local surface. Since the curves at that local 
surface are disjoint. We could have only added 
a single face incident with $e$. We take for $\sigma''(e)$ the unique 
cyclic orientation of size two at $e$. 

So from now assume that $e$ is incident with at least two faces of $C'$. 
In order to define $\sigma''(e)$, we start with $\sigma'(e)$ and define in the following for each 
newly added face in 
between which two cyclic orientations of faces adjacent in $\sigma'(e)$ we put it. We shall ensure 
that between any two orientations we put 
at most one new face. Recall that two cyclic orientations $\vec{f_1}$ and $\vec{f_2}$ of faces 
$f_1$ and $f_2$, respectively, are adjacent 
in $\sigma'(e)$ if and only if there is a clone $e'$ of $e$ in a local surface $S$ for 
$(C',\Sigma')$ containing $\vec{f_1}$ and 
$\vec{f_2}$ such that $e'$ is incident with $\vec{f_1}$ and $\vec{f_2}$ in $S$.
Let $f$ be a face newly added to $C''$ at $e$.  Let $\gamma_f$ be the curve from which $f$ is build 
and let $S_f$ be the local surface 
that has $\gamma_f$ in its curve system. Let $e_f$ be the copy of $e$ in $S_f$ that corresponds to 
$f$ as defined above.  when we 
consider $f$ has a face obtained from the disc glued at  $\gamma_f$.
We add $f$ to $\sigma'(e)$ in between the two cyclic orientations that are incident with $e_f$ in 
$S_f$. 
This completes the definition of $\Sigma''$.
Since the copies $e_f$ are distinct for different faces $f$, the rotation system $\Sigma''$ is 
well-defined.
By construction $\Sigma''$ induces $\Sigma$. We prove the following.

\begin{sublem}\label{primeprime}
 $\Sigma''$ is a planar rotation system of $C''$. 
\end{sublem}

\begin{proof}
Let $v$ be a vertex of $C''$. If $v$ is not a vertex of $C'$, then the link graph at $v$ 
is a cycle. Hence the link complex at $v$ is clearly a sphere. Hence we may assume that $v$ is a 
vertex of $C'$. 

Our strategy to show that the link complex $S''$ at $v$ for $(C'',\Sigma'')$ is a sphere 
will be to show that it is obtained 
from the link complex $S'$ for $(C',\Sigma')$ by adding edges in such a way that each 
newly added edge traverses a face of $S'$ 
and two newly added edges traverse different face of $S'$. 

So let $f$ be a newly added face incident with $v$ of $C'$. 
Let $x$ and $y$ be the two edges of $f$ incident with $v$. 
We make use of the notations $\gamma_f$, $S_f$, $x_f$ and $y_f$ defined above. 
Let $v_f$ be the unique vertex of $S_f$ traversed by $\gamma_f$ in between $x_f$ and $y_f$.
By \autoref{loc_inc_AND_loc_surfaces} there is a unique face $z_f$ of $S'$ mapped by the map 
$\iota$ of that lemma to 
$v_f$. 
And $x$ and $y$ are vertices in the boundary of $z_f$. The edges on the boundary of $z_f$ incident 
with $x$ and $y$ are the cyclic 
orientations of the faces that are incident with $x_f$ and $y_f$ in $S_f$. Hence in $S''$ the edge 
$f$ traverses the face $z_f$.

It remains to show that the faces $z_f$ of $S'$ are distinct for different newly added faces $f$ of 
$C''$. For that it suffices by 
\autoref{loc_inc_AND_loc_surfaces} to show that the vertices $v_f$ are distinct. This is true as 
curves for $S_f$ traverse a vertex of $S_f$ 
at most once and different curves for $S_f$ are disjoint. 

\end{proof}

 Since $\Sigma''$ is a planar rotation system of the locally connected simplicial complex $C''$ and 
$C''$ is simply connected, $\Sigma''$ is 
induced by a topological embedding of 
$C''$ into \Sthree\ by \autoref{combi}. Hence 2 implies 3.
\end{proof}

A natural weakening of the property that $C$ is simply connected is that the fundamental group of 
$C$ is abelian. 
Note that this is equivalent to the condition that every chain that is $p$-nullhomologous is 
simply 
connected.

\begin{thm}\label{abelian_alt}
 Let $C$ be a connected and locally connected simplicial complex with abelian fundamental group. 
Then $C$ has a topological embedding 
into $\Sbb^3$ if 
and only if it has a planar rotation system $\Sigma$ that has a $p$-nullhomologous extension. 
\end{thm}

In order to prove \autoref{abelian_alt}, we prove the following.

\begin{lem}\label{abi}
 A $p$-nullhomologous extension of a planar rotation system of a simplicial complex $C$ with 
abelian 
fundamental group is a simply connected 
extension. 
\end{lem}

\begin{proof}
 Let $C'$ be the topological space obtained from $C$ by gluing discs along the curves of the 
$p$-nullhomologous extension. The fundamental 
group $\pi'$ of $C'$ is a quotient of the fundamental group $\pi$ of $C$, see for example 
{\cite[Proposition 1.26]{Hatcher}}.
Since $\pi$ is abelian by assumption, also $\pi'$ is abelian. That is, it is equal to its 
abelisation, which is trivial by assumption.
Hence $C'$ is simply connected.
\end{proof}

\begin{proof}[Proof of \autoref{abelian_alt}.]
 If $C$ has a topological embedding into $\Sbb^3$, then by \autoref{general} it has a planar 
rotation system that has a $p$-nullhomologous 
extension.
 If $C$ has a planar rotation system that has a $p$-nullhomologous 
extension, then that extension is simply connected by \autoref{abi}. Hence $C$ has a topological 
embedding into $\Sbb^3$ by the other 
implication of \autoref{general}. 
\end{proof}

\section{Non-orientable 3-manifolds}\label{non-or}

By \autoref{prs_orient} and \autoref{is_manifold} a 2-complex $C$ is embeddable in an 
orientable 3-manifold if and only if it has a planar rotation system. 
Here we discuss a notion similar to planar rotation systems that can be used to characterise 
embeddability in 3-manifolds combinatorially.

A \emph{generalised planar rotation system} of a 2-complex $C$ is a choice of embedding of all its 
link graphs in the plane so that for each edge $e$ of $C$ in the two link graphs at the endvertices 
of the edge $e$, the rotators at the vertex $e$ are reverse or agree. We colour an edge \emph{red} 
if they agree. We have the further condition that for every face $f$ of $C$ its number of red 
incident edges is even. 

\begin{rem}
Any planar rotation system defines a generalised planar rotation system. In this induced 
generalised planar rotation system no edge is red. 
\end{rem}

\begin{rem}
Let $C$ be simplicial complex whose first homology group over the binary field $\Fbb_2$ is trivial, 
for example this includes the case where $C$ is simply connected. Then $C$ has a planar rotation 
system if and only if it has a generalised planar rotation system by \autoref{nullt} and 
{\cite[\autoref*{pre-rot}]{3space1}}.
\end{rem}

\begin{rem}\label{trivial}
A locally 3-connected 2-complex has a generalised planar rotation system if and only if every choice 
of embeddings of its link graphs in the plane is a generalised planar rotation system. Indeed, any 
face shares an even number of edges with every vertex. By local 3-connectedness, the embeddings of a 
link graph at some vertex $v$ is unique up to orientation, which reverses all rotators at edges 
incident with $v$. So it flips red edges to non-red ones and vice versa. Hence the number of red 
edges modulo two does not change. 

Thus there is a trivial algorithm that verifies whether a locally 3-connected 2-complex has a 
generalised planar rotation system: embed the link graphs in the plane arbitrarily (if possible) and 
then check the condition at every edge and face. 
\end{rem}

The following was proved by Skopenkov.

\begin{thm}[\cite{Skopenkov94}]
The following are equivalent for a 2-complex $C$.
\begin{itemize}
\item $C$ has a generalised planar rotation system;
\item $C$ has a thickening that is a 3-manifold;
\item $C$ has a piece-wise linear embedding in some 3-manifold.
\end{itemize}
\end{thm}

\begin{rem}
For 2-dimensional manifolds there are two types of rotation systems: one for the orientable case 
and one for the general case. In dimension three the situation is analoguous. For orientable 
3-manifolds we have `planar rotation systems' and `generalised planar 
rotation systems' for general 3-manifolds. 
\end{rem}

\begin{rem}
Similarly as for planar rotation systems, the methods of this series of papers can also be used to 
also give a polynomial time algorithm that verifies the existence of generalised planar rotation 
systems. As mentioned in \autoref{trivial} the locally 3-connected case is algorithmically trivial 
and the reduction to the locally 3-connected case works the same as for planar rotation systems. 
\end{rem}

\begin{rem}
The methods of this series of papers can be used to 
characterise the existence of generalised planar rotation systems in terms of excluded minors as 
follows.

In the locally 3-connected case the list of excluded minors for generalised planar rotation 
systems is slightly shorter than for planar rotation systems. It still contains 
the obstructions in the list $\Zcal_1$ like the cone over $K_5$, and also more complicated 
obstructions such as that explained in {\cite[\autoref*{other_construction}]{3space1}}. But from 
the list $\Zcal_2$ this list just contains those obstructions  
that are derived from those in {\cite[\autoref*{equal}]{3space1}} where the vertices $v$ and 
$w$ are joined by an edge.

Indeed arguing as in the proof of {\cite[\autoref*{rot_system_exists}]{3space1}}, we find an 
obstruction sitting either at a vertex, an edge or a 
face. In the first two cases we just get obstructions in $\Zcal_1$. In the third case we can 
contract two edges of that face and we only get those obstructions $C\in\Zcal_2$ such that the 
loop 
$\ell$ bounds a face. In this case in the strict marked graph at 
the link graph of the contraction vertex of $C$,  the vertices $v$ and $w$ corresponding to the 
loop are joined by an edge.

The other obstructions in the list $\Zcal_2$ are still necessary obstructions for embeddability in 
the 3-sphere  but the above shows that under the assumption of simply connectedness and local 
3-connectedness, we do not need them to characterise embeddability in the 3-sphere.

In the general case, however, generalised planar rotation systems need also some new obstructions 
compared to planar rotation systems. Indeed, instead of the torus crossing obstructions we now get 
the following.

It is straightforward to check that the torus crossing obstructions, where the pair (a,b) of 
winding numbers is 
different from (1,2), still give obstructions. Thus it remains to understand the obstructions 
arising from para-cycles all of whose winding numbers are one or two. 
Here additionally we have Klein Bottle crossing 
obstructions. 
These are defined similarly as torus crossing obstructions but here we partition the set 
into more than two equivalence classes. These new equivalence classes have winding numbers (1,2,2) 
or (1,1,1,2). 

Indeed, given a para-cycle all of whose winding numbers are one or two, we take a single letter for 
each equivalence class of winding number one, and for each equivalence class of size two we take a 
pair of letters $(n,n')$. 
Then the para-cycle has a generalised planar rotation system if and only if its letters form a word 
that uses each letter precisely once such that if we exchange $n$ and $n'$ for every 
pair, the resulting word is equal to the original word or its reverse.
For (1,2) there is such a word, namely: $1X1'$. For (1,1,2), such a word is $Y1X1'$.
So (1,2) and (1,1,2) do not lead to obstructions. All other families 
of winding numbers from $\{1,2\}$ of size at least three do lead to obstructions. For that it 
suffices to show that for the (minimal) families (1,2,2) and (1,1,1,2) there are no such words. 
This is an 
easy exercise\footnote{To give details, for (1,2,2) we have the letters $X,1,1',2,2'$. There is a 
unique word satisfying the requirements if we delete 
the letter $X$. It is $121'2'$, up to symmetry. It is straightforward to check that this word 
cannot be extended to a legal word by adding $X$. 

For (1,1,1,2) the letters are $X,Y,Z,1,1'$. By symmetry we may assume that $XY1$ is a subinterval. 
But this is not contained in any legal word.}

To summarise, for the Kuratowski Theorem characterising 2-complexes with generalised planar 
rotation systems the list of excluded minors is the following.
It consists of the slightly shorter list for the locally 3-connected case described above together 
with the torus crossing 
obstructions except for (1,2). Additionally we have the Klein Bottle crossing obstructions (1,2,2) 
and (1,1,1,2). 
\end{rem}

   \section*{Acknowledgement}

I thank Arnaud de Mesmay for useful discussions that led to \autoref{non-or}.

\bibliographystyle{plain}
\bibliography{literatur}

\end{document}